\documentclass[11pt]{article}
\usepackage{amsmath,amsthm,amsfonts,amssymb,mathrsfs,bm,graphicx,stmaryrd,hyperref}
\usepackage{mathtools}
\usepackage[usenames]{color}
\usepackage[letterpaper,hmargin=1.0in,vmargin=1.0in]{geometry}
\parskip 	\smallskipamount
\usepackage{tocloft}
\cftsetindents{section}{-1em}{2em}
\newtheorem{theorem}{Theorem}[section]

\newtheorem{corollary}[theorem]{Corollary}

\def\P{\mathbb{P}}

\def\R{\mathbb{R}}

\begin{document}
\title{A note on the local weak limit of a sequence of expander graphs}

\author{
Sourav Sarkar
\thanks{Department of Mathematics, University of Toronto. Email: sourav.sarkar@utoronto.ca}
}

\date{}
\maketitle

\date{}
\maketitle
           
\begin{abstract}We show that the local weak limit of a sequence of finite expander graphs with uniformly bounded degree is an ergodic (or extremal) unimodular random graph. In particular, the critical probability of percolation of the limiting random graph is constant almost surely. As a corollary, we obtain an improvement to a theorem by Benjamini-Nachmias-Peres (2011) in \cite{BNP} on locality of percolation probability for finite expander graphs with uniformly bounded degree where we can drop the assumption that the limit is a single rooted graph.
\end{abstract}
\section{Introduction}
Local weak convergence of a sequence of finite expander graphs has been studied in relation to locality of critical probability of percolation in \cite{BNP}. Locality of critical probability for Bernoulli bond percolation on infinite graphs is in itself a very significant and well studied problem, beginning with Schramm's locality conjecture for transitive graphs; see \cite{MT} and  \cite{BPT} and the recent paper \cite{T}  and the references therein for some recent developments in this direction for infinite graphs. For finite graphs, such locality was shown to hold when the graphs were expanders and the sequence of graphs converged locally weakly to a fixed infinite rooted graph in \cite{BNP}. Their argument, in turn, was an extension of the arguments used in \cite{ABS}, who studied critical probability for the emergence of a giant component (that is a connected linear sized component) in finite expander graphs under the
assumptions of regularity and high-girth. In \cite{BBLR} it was shown that in any expander, every giant component of given proportion emerges
in an interval of length $o(1)$, that is,  for any $c \in (0,1)$, the property that the random subgraph of an expander $G=(V,E)$ after Bernoulli bond percolation contains a giant component of size $c|V|$, has a sharp threshold; removing
the regularity and high-girth assumptions in \cite{ABS}.

We first recall local weak convergence of bounded degree finite graphs. A rooted graph is denoted as $(G,o)$, where $G$ is a (connected) graph
and $o$ is a vertex in $G$. A rooted graph $(G,o)$ is isomorphic to $(G',o')$, written as $(G,o)\cong(G',o')$ if
there is an isomorphism of $G$ onto $G'$ which takes $o$ to $o'$. Let $\mathcal{G}$ denote the space of isomorphism classes of all rooted connected graphs with degrees bounded by $\Delta$, for some fixed $\Delta>0$.  For $r=1,2,\ldots$ and a graph $G$, let $B_G(o,r)$ denote the closed ball of radius $r$ around the vertex $o$ in the graph $G$. Define the distance between two isomorphic classes of rooted graphs $(G_1,o_1)$ and $(G_2,o_2)$ as $D((G_1,o_1),(G_2,o_2))=\frac{1}{1+t}$ where $t=\sup\{s:B_{G_1}(o_1,s)\cong B_{G_2}(o_2,s) \}$. The space $(\mathcal{G},D)$ is a compact and separable metric space.  

For each
$n \geq 1$, let $G_n$ be a finite graph and let $U_n$ be a uniformly chosen random vertex in $G_n$ and $(G, \rho)$ be an infinite random rooted
graph in $\mathcal{G}$, that is, $(G,\rho)$ is a sample from a Borel probability measure $\mu$ on $\mathcal{G}$.
We say that the sequence of finite graphs $\{G_n\}$ \textit{converges locally weakly} to $(G, \rho)$ (or to $\mu$) if
for every $R > 0$ and for every finite rooted graph $(H, \rho')$, 
\[\P\left((B_{G_n} (U_n , R), U_n\right)\cong (H,\rho'))\rightarrow
\P\left((B_G(\rho, R), \rho)\cong (H,\rho')\right)\mbox{\ \ \ as } n\rightarrow \infty\,.\]
 This is equivalent to saying that $\mu$ is the weak limit of the
laws of $(G_n , U_n)$. This is a special case of the graph limits defined in \cite{BS}. Such measures $\mu$ satisfy a ``spatial stationarity" property called \textit{unimodularity}, which roughly says that if mass is redistributed in the graph,
then the expected mass that leaves the root is equal to the expected mass the arrives at the root, see Definition $2.1$ of \cite{AL} for an exact definition.

Also, we shall denote by $C((G,o),\frac{1}{1+s})$ the $\frac{1}{1+s}$ neighborhood of the graph $(G,o)$ in the space $(\mathcal{G},D)$, that is, the set of all rooted graphs whose $s$-balls are isomorphic to $(B_G(o,s),o)$.

Let $\mathcal I$ denote the $\sigma$-field of events in the Borel $\sigma$-field of $\mathcal G$ that are invariant under non-rooted isomorphisms. The class $\mathcal U$ of unimodular probability measures on $\mathcal G$ is convex. An element of $\mathcal U$ is
called \textbf{extremal} if it cannot be written as a convex combination of other elements of $\mathcal U$. It follows from Theorem $4.7$ of \cite{AL} that a unimodular probability measure $\mu$ on $\mathcal G$ is extremal if and only if $\mathcal I$ is $\mu$-trivial, that is, $\mathcal I$ contains only sets of $\mu$-measure $0$ or $1$ (ergodicity).

For two sets of vertices $A$ and $B$,
we shall write $E(A, B)$ for the set of edges with one endpoint in $A$ and the other in $B$. Recall
that the Cheeger constant or the edge-isoperimetric number $h(G)$ of a finite graph $G = (V, E)$ is defined by
\[h(G) = \min_{
A\subseteq V}\left\{\frac{
|E(A, V \setminus A)|}{
|A|}
: 0 < |A| \leq  |V|/2
\right\}\,.\]

Now, we recall Theorem $1.3$ of \cite{BNP} here. Let $(G, \rho)$ be a \textbf{fixed} infinite bounded degree rooted graph and $p_c(G):=\inf\{p\in [0,1]:\P_p(\exists \mbox{ an infinite open component})>0\}$ where $\P_p$ denotes Bernoulli bond percolation with probability $p$, that is, each edge is independently open with probability $p$ and closed with probability $1-p$. Let $G_n$ be
a sequence of finite graphs with a uniform Cheeger constant lower bound $c > 0$ and a
uniform degree bound $\Delta$, such that $G_n \rightarrow (G,\rho)$ locally weakly. Let $p \in [0, 1]$ and write $G_n(p)$
for the graph of open edges obtained from $G_n$ by performing bond percolation with
parameter $p$. If $p < p_c(G)$, then for any constant $\alpha > 0$ we have
\[\P (G_n(p) \mbox{ contains a component of size at least } \alpha|G_n|) \rightarrow 0 \mbox{ \ \ as } n\rightarrow \infty,\]
and if $p > p_c(G)$, then there exists some $\alpha = \alpha(p) > 0$ such that
\[\P (G_n(p) \mbox{ contains a component of size at least } \alpha|G_n|) \rightarrow 1 \mbox{ \ \ as } n\rightarrow \infty\,.\]

In this short paper, we first show that when a sequence of finite expander graphs converges locally weakly to a random graph $G$, then $G$ is an ergodic (or extremal) unimodular random graph. Any measurable function $f: \mathcal{G}\mapsto \R$ is called rerooting-invariant if it is invariant under changes in the position of the root (that is, for any graph $\tau$ and any two vertices $v,v'\in \tau$, $f((\tau,v))=f((\tau,v'))$; we sometimes suppress the root in the notation and simply denote it as $f(\tau)$). We have the following theorem.
\begin{theorem}\label{t} 
Let $G_n$ be a sequence of finite graphs with a uniform Cheeger constant lower bound $c>0$ and a uniform degree bound $\Delta>0$, such that $G_n\rightarrow (G,\rho)$ locally weakly, where $(G,\rho)$ is a random infinite rooted graph. Then $G$ is an ergodic (or extremal) unimodular random graph. That is, if $f$ is any rerooting-invariant function, then $f(G)$ is constant almost surely. In particular,  $p_c(G)$ is constant almost surely.

\end{theorem}

As a corollary to the above theorem, we get the following improvement to Theorem 1.3 of \cite{BNP}.
\begin{corollary}\label{c}Let $G_n$ be a sequence of finite expander graphs with uniformly bounded degree as above, such that $G_n$ converges locally weakly to an infinite \textbf{random} rooted graph $(G,\rho)$. If $p < p_c(G)$ (this is well defined as $p_c(G)$ is constant by Theorem \ref{t}), then for any constant $\alpha > 0$ we have
\[\P (G_n(p) \mbox{ contains a component of size at least } \alpha|G_n|) \rightarrow 0 \mbox{ \ \ as } n\rightarrow \infty,\]
and if $p > p_c(G)$, then there exists some $\alpha = \alpha(p) > 0$ such that
\[\P (G_n(p) \mbox{ contains a component of size at least } \alpha|G_n|) \rightarrow 1 \mbox{ \ \ as } n\rightarrow \infty\,.\]
\end{corollary}

\section{Proofs}

We first prove Theorem \ref{t}. 
\begin{proof}[Proof of Theorem \ref{t}] Without loss of generality we assume that $|G_n|=n$. Let the law of the limiting random rooted graph $(G,\rho)$ be denoted by $\mu$ and $f$ be any rerooting-invariant function.
We show that for any two rational numbers $a,b$ such that $0\leq a<b\leq 1$, if 
\[\Gamma_1=\{(\tau,r)\in \mathcal{G}: f(\tau)\leq a\}, \ \ \ \Gamma_2=\{(\tau,r)\in \mathcal{G}: f(\tau)\geq b\}\,,\]
then $\mu(\Gamma_1)$ and $\mu(\Gamma_2)$ cannot be both positive. Clearly, this is enough to prove the theorem. We prove this by contradiction. Let $a,b$ be rational numbers, $p_0>0$ be some real number and $\Gamma_i$'s be defined as above, such that $\mu(\Gamma_i)\geq p_0$ for $i=1,2$. Now, as $(\mathcal{G,D})$ is a compact metric space, and $\mu$ is a probability measure, hence $\mu$ is \textit{regular} (see, for example, \cite{KRP} Chapter II, Theorem $1.2$). Hence there exist compact sets $H_i\subseteq \Gamma_i$ such that $\mu(H_i)\geq p_0/2$. Fix
\[K=\frac{4\Delta}{cp_0}.\]

For any graph $(\tau,r)\in \mathcal{G}$, let $[\tau]=\{(\tau,r'):r'\in \tau\}$ denote the unrooted version of $(\tau,r)$. Now, for any fixed $(\tau_i,r_i)\in H_i$, since $f((\tau_1,r_1))\neq f((\tau_2,r_2))$, hence $[\tau_1]\neq [\tau_2]$. Define
\begin{eqnarray*}
R_{(\tau_1,r_1),(\tau_2,r_2)}:=\max_{u_i\in B_{\tau_i}(r_i,K)} \{R_{u_1,u_2}: R_{u_1,u_2} \mbox{ is the smallest $s\in \{1,2,\ldots\}$ for which } \\ B_{\tau_1}(u_1,s)\not \cong B_{\tau_2}(u_2,s)\}.
\end{eqnarray*}
Since $[\tau_1]\neq [\tau_2]$, $R_{u_1,u_2}<\infty$. As $(\tau_i,r_i)$ have uniformly bounded degree, the maximum is taken over a finite set, so that $R_{(\tau_1,r_1),(\tau_2,r_2)}<\infty$.

Consider the neighborhood $C((\tau_2,r_2),\frac{1}{1+R_{(\tau_1,r_1),(\tau_2,r_2)}+K})$, that is, the set of all rooted graphs whose $(R_{(\tau_1,r_1),(\tau_2,r_2)}+K)$-balls are isomorphic to $(B_{\tau_2}(r_2,R_{(\tau_1,r_1),(\tau_2,r_2)}+K),r_2)$. Further, consider \[\mathfrak C:=\bigcup_{(\tau_2,r_2)\in H_2}C\left((\tau_2,r_2),\frac{1}{1+R_{(\tau_1,r_1),(\tau_2,r_2)}+K}\right)\,.\]
 Since $\mathfrak C$ is an open cover of $H_2$ and $H_2$ is compact, it has a finite subcover, that is, for some $M>0$,
\[H_2\subseteq \bigcup_{i=1}^M C\left((\tau_{2,i},r_{2,i}),\frac{1}{1+R_{(\tau_1,r_1),(\tau_{2,i},r_{2,i})}+K}\right)\,.\] 
  Let \[R_{(\tau_1,r_1)}=\max \left\{R_{(\tau_1,r_1),(\tau_{2,i},r_{2,i})}+K:i\in \{1,2,\ldots,M\}\right\}\] over this finite subcover. 
  
  Moreover, as 
  \[\bigcup_{(\tau_1,r_1)\in H_1}C\left((\tau_1,r_1),\frac{1}{1+R_{(\tau_1,r_1)}}\right)\]
   is an open cover of $H_1$ and $H_1$ is compact, it has a finite subcover, and let $R=\max_i R_{(\tau_{1,i},r_{1,i})}$ be the maximum as $\{(\tau_{1,i},r_{1,i})\}$ range over this finite subcover. Then for all $(\tau_i,r_i)\in H_i$, 
\[B_{\tau_1}(u_1,R)\neq B_{\tau_2}(u_2,R) \quad \forall u_i\in B_{\tau_i}(r_i,K) \mbox{ for } i=1,2.\]

Define for $i=1,2$ and $G_n=(V_n,E_n)$,
\[A_{i,n}=\{v\in V_n: B_{G_n}(v,R+K)\cong B_{\tau_i}(v_i,R+K) \mbox{ for some } (\tau_i,v_i)\in H_i\}.\]

Choose and fix $n$ large enough such that $|A_{i,n}|\geq \frac{np_0}{4}$ and let $A_i=A_{i,n}$.

Now, recall the following theorem due to Menger (1927). \textit{Let $G=(V,E)$ be a graph and $A,B\subseteq V$. Then the maximum number of edge disjoint $A-\mbox{to}-B$ paths equals the minimum size of an $A-\mbox{to}-B$ separating edge cut.} 

Let $L$ denote the maximum number of edge disjoint paths between $A_1$ and $A_2$, and $E(S,S^c)$ denote the number of edges between $S$ and $S^c$. Applying Menger's theorem with the sets $A_1,A_2$, and using $|A_i|\geq np_0/4$, and the fact that the Cheeger constant is bounded below by $c>0$ from the assumption of our theorem, we get
\[\frac{L}{np_0/4}=\min_{S\subseteq V, A_1\subseteq S, A_2\subseteq S^c}\frac{E(S,S^c)}{np_0/4}\geq \min_{S\subseteq V, A_1\subseteq S, A_2\subseteq S^c}\frac{E(S,S^c)}{\min(|S|,|S^c|)}\geq \min_{S\subseteq V}\frac{E(S,S^c)}{\min(|S|,|S^c|)}\geq c.\]

Hence, there are at least $cnp_0/4$ edge-disjoint paths from $A_1$ to $A_2$. Since there are at most $\Delta n/2$ edges in $G$, at least half of these paths, i.e., at least $cnp_0/8$ edge-disjoint paths between $A_1$ and $A_2$ have length at most $K$ each. Take two vertices $v_i\in A_i$ that are incident on such a path of length at most $K$. Since \[B(v_2,R)\subseteq B(v_1,R+K)\cong B_{\tau_1}(r_1,R+K)\] 
for some $(\tau_1,r_1)\in H_1$, and \[B(v_2,R)\cong B_{\tau_2}(r_2,R)\] for some $(\tau_2,r_2)\in H_2$; hence $B_{\tau_1}(u_1,R)\cong B_{\tau_2}(r_2,R)$ for some $u_1$ such that $d(r_1,u_1)\leq K$. This contradicts the choice of $R$. This completes the proof.
\end{proof}

Finally we end this section by proving Corollary \ref{c}.
\begin{proof}[Proof of Corollary \ref{c}] We first prove the case when $p<p_c(G)$. Fix any $\varepsilon>0$. Let $H\subseteq \mathcal{G}$ be the set of all rooted graphs with percolation critical probability equal to $p_c(G)$. For every $(\tau,v) \in H$, as $p<p_c(\tau)$, there exists $R_{(\tau,v)}=R_{(\tau,v)}(\varepsilon)$ large enough, such that,
\begin{equation}\label{e:bdrcon}
\P_p(v \leftrightarrow \partial B_\tau(v,R_{(\tau,v)}))<\varepsilon.
\end{equation}
Consider the neighborhood $C((\tau,v),\frac{1}{1+R_{(\tau,v)}})$. Then for all graphs in this neighborhood, their $R_{(\tau,v)}$ balls around the root are isomorphic to $B_\tau(v,R_{(\tau,v)})$, and hence the equation \eqref{e:bdrcon} holds for all such graphs. Consider $\cup_{(\tau,v)\in H}C((\tau,v),\frac{1}{1+R_{(\tau,v)}})$, an open cover of $H$. Also $H$ is second countable (since $\mathcal{G}$ is). Hence $H$ admits a countable subcover, say $C_1,C_2,\ldots$, where $C_i=C((\tau_i,v_i),\frac{1}{1+R_{(\tau_i,v_i)}})$ for some $(\tau_i,v_i)\in H$. Let $H'=\cup_{i=1}^M C_i$ for some large integer $M$ be such that $\mu(H')\geq 1-\varepsilon$, where $\mu$ is the law of the limiting random rooted graph. Let $R=\max_{i=1,2,\ldots,M}R_{(\tau_i,v_i)}$. Then for all $(\tau,v)\in H'$,
\[\P_p(v \leftrightarrow \partial B_\tau(v,R))<\varepsilon.\]

Now as $G_n$ converges locally weakly to $G$, choose $n$ large enough such that, if $F$ denotes the finite set of all rooted graphs of radius $R$ and degree bounded by $\Delta$, then,
\begin{equation}\label{eq:sumlim}
\sum_{(\tau,v)\in F}\left|\mathcal{L}(B_{G_n}(U_n,R)=(\tau,v))-\P(B_G(\rho,R)=(\tau,v))\right|\leq \varepsilon\,,
\end{equation}
where $\mathcal{L}$ is the law of $U_n$.
Hence,
\begin{eqnarray}\label{eq:upbnd}
&&\mathcal{L}\times \P_p(U_n\leftrightarrow \partial B_{G_n}(U_n,R))\\
&=&\sum_{(\tau,v)\in F}\mathcal{L}\times \P_p\Big(U_n\leftrightarrow \partial B_{G_n}(U_n,R)\big|B_{G_n}(U_n,R)=(\tau,v)\Big)\mathcal{L}(B_{G_n}(U_n,R)=(\tau,v))\nonumber\\
&=&\sum_{(\tau,v)\in F} \P_p(v\leftrightarrow \partial B_\tau(v,R))\mathcal{L}(B_{G_n}(U_n,R)=(\tau,v))\nonumber\\
&\leq & \sum_{(\tau,v)\in F} \P_p(v\leftrightarrow \partial B_{\tau}(v,R))\P(B_G(\rho,R)=(\tau,v))+\varepsilon\nonumber\\
&\leq &\sum_{B_\tau(v,R):(\tau,v)\in H'} \P_p(v\leftrightarrow \partial B_{\tau}(v,R))\P(B_G(\rho,R)=B_\tau(v,R))+2\varepsilon\nonumber\\
&\leq & \varepsilon \mu(H')+2\varepsilon\leq 3\varepsilon\,,\nonumber
\end{eqnarray}
where the inequality in the fourth line follows from \eqref{eq:sumlim}.
Next, following the arguments in the proof of Theorem $1.3$ in \cite{BNP} verbatim, it follows that 
\[\P_p(|C_1(n)|\geq \alpha n)\leq 3\varepsilon\alpha^{-1},\]
where $C_1(n)$ is the largest component of $G_n(p)$, which proves the first assertion of the corollary.

Now we prove the case when $p>p_c(G)$. As in the proof of Theorem $1.3$ of \cite{BNP} for this case, fix some $\varepsilon>0$, let $p_1>p_c$ such that $1-p=(1-p_1)(1-\varepsilon)$, and consider $G_n(p_1)$. As $p_1>p_c(G)$, for all $(\tau,v)\in H$, the set of all rooted graphs with percolation probability equal to $p_c(G)$, we have,
\[f((\tau,v)):=\P_{p_1}(v\leftrightarrow\infty):=\inf_{R}\P_{p_1}(v\leftrightarrow\partial B_\tau(v,R))>0.\]
As $\mu\{\cup_{m=1}^\infty\{(\tau,v)\in H:f((\tau,v))>\frac{1}{m}\}\}=\mu\{(\tau,v)\in H:f((\tau,v))>0\}=1$, hence there exists some $\delta>0$  and $0<\eta\leq 1$ such that
\[\mu\{(\tau,v)\in H:f((\tau,v))>\delta\}\geq \eta>0.\]
Let $H''\subseteq H$ be the set of all rooted graphs $(\tau,v)$ such that $f((\tau,v))>\delta$. Then $\mu(H'')\geq \eta >0$. 

Fix $R>(\varepsilon^{\frac{12\Delta}{c\eta\delta}}c\eta\delta/24)^{-1}$ as in the proof of Theorem $1.3$ of \cite{BNP}. Then for this $R$, there exists $n_0$, such that for $n\geq n_0$, we have in $G_n$, (using \eqref{eq:sumlim} and setting $\varepsilon=\delta\eta/2$)
\begin{eqnarray*}
&&\mathcal{L}\times \P_{p_1}(U_n\leftrightarrow \partial B_{G_n}(U_n,R))\\
&\geq& \sum_{B_\tau(v,R):(\tau,v)\in H} \P_{p_1}(v\leftrightarrow \partial B_{\tau}(v,R))\P(B_G(\rho,R)=B_\tau(v,R))-\delta\eta/2\\
&\geq &\sum_{B_\tau(v,R):(\tau,v)\in H''} \P_{p_1}(v\leftrightarrow \partial B_{\tau}(v,R))\P(B_G(\rho,R)=B_\tau(v,R))-\delta\eta/2\\
&\geq & \delta \mu(H'')-\delta\eta/2\geq \delta\eta-\delta\eta/2=\delta\eta/2.
\end{eqnarray*}
For $v\in G_n$, let $B'_{p_1}(v,R)$ denote the set of vertices in $G_n(p_1)$ which are connected to $v$ in a $p_1$-open path of length at most $R$. Thus for all $n\geq n_0$,
\[\mathcal{L}\times \P(|B'_{p_1}(U_n,R)|\geq R)\geq\mathcal{L}\times \P_{p_1}(U_n\leftrightarrow \partial B_{G_n}(U_n,R))\geq \delta\eta/2. \] 
Next following the arguments in the proof of Theorem $1.3$ of \cite{BNP} verbatim, the lemma follows.
\end{proof}

\section{Acknowledgement} This work was completed when the author was a graduate student at UC Berkeley and was supported by the Lo\`{e}ve Fellowship, which he gratefully acknowledges. The author thanks Nike Sun for suggesting this problem and for some helpful discussions. He also thanks Tom Hutchcroft for very helpful feedback on an earlier version of the paper. The author thanks the anonymous referee whose careful reading and detailed comments helped improve the paper.

\bibliography{Exp}
\bibliographystyle{amsplain}

\end{document}